\documentclass{amsart}
\usepackage{amssymb}
\usepackage{amsmath}
\usepackage{amsthm}
\usepackage{amscd}
\usepackage{mathrsfs}
\usepackage{bm}
\usepackage{hyperref}
\DeclareMathOperator{\Ima}{Im}
\DeclareMathOperator{\dom}{dom}
\newcommand{\N}{\mathbb{N}}

\newcommand{\C}{\mathbb{C}}
\newcommand{\Z}{\mathbb{Z}}

\theoremstyle{definition}
\newtheorem{dfn}{Definition}[section]
\theoremstyle{plain}
\newtheorem{lm}[dfn]{Lemma}
\newtheorem{pr}[dfn]{Proposition}
\newtheorem{cor}[dfn]{Corollary}
\newtheorem{thm}[dfn]{Theorem}
\theoremstyle{remark}
\newtheorem{rem}[dfn]{Remark}

\title{A characterization of rationality in free semicircular operators}
\author{Akihiro Miyagawa}
\address{Department of Mathematics, Kyoto University, Kitashirakawa Oiwake-cho, Sakyo-ku, 606-8502, Japan}
\email{miyagawa.akihiro.43v@st.kyoto-u.ac.jp}

\begin{document}
\maketitle
\begin{abstract}
For free semicircular elements realized on the full Fock space, we prove an equivalence between rationality of operators obtained from them and finiteness of the rank of their commutators with right annihilation operators. This is an analogue of the result for the reduced $C^*$-algebra of the free group by G. Duchamp and C. Reutenauer which was extended by P. A. Linnell to densely defined unbounded operators affiliated with the free group factor. While their result was motivated by quantized calculus in noncommutative geometry, we state our results in terms of free probability theory.       
\end{abstract}
\section{Introduction}

In 1881, Kronecker \cite{Kro1881} found an interesting connection between Hankel matrices and rational functions. In terms of functional analysis, his result shows the equivalence between finite rank Hankel operators and bounded rational functions on the unit circle (see also Section \ref{kronecker's theorem}). This connection is non-trivial and relies on a beautiful combination of analytic estimates and purely algebraic properties.    

The argument of Kronecker's theorem also appears in quantized calculus for noncommutative geometry. A. Connes conjectured an analogue of this theorem for the reduced free group $C^*$-algebra $C^*_{\mathrm{red}}(\mathbb{F}_d)$ in his book \cite[Section 4.5]{Connes94}. 
He considered the commutator $[F, \cdot]$ where $F$ arises from a free action of the free group on a tree, and it can be viewed as an analogue of bounded Hankel operators. His conjecture was that rationality of $a$, for any elements in $C^*_{\mathrm{red}}(\mathbb{F}_d)$, is equivalent to finiteness of the rank of $[F,a]$. This conjecture was solved by G. Duchamp and C. Reutenauer \cite{DR97} and extended by P. A. Linnell \cite{Lin00} to densely defined closed operators affiliated with the free group factor. Their proof is based on facts from noncommutative rational series which are also related to Hankel matrices indexed by the free monoid (see \cite[Theorem 2.1.6]{BR10}).     

In this paper, we prove a similar phenomenon for a $d$-tuple of free semicircular elements, instead of the generators of the free group. In free probability theory, free semicircular elements are of central importance. They are limit objects not only for free central limit theorem but also for empirical eigenvalue distributions of independent Gaussian unitary ensembles which are typical random matrix models (we recommend \cite{VDN92}, \cite{NS06}, and \cite{MS17} for the textbooks of free probability theory). 
Via the GNS-representation, they are represented on the full Fock space $\mathcal{F}(H)$ by the left annihilation and creation operators. While the operator $F$ has a role to determine rationality of elements in the reduced free group $C^*$-algebra, we prove that the right annihilation operators $\{r_i^*\}_{i=1}^d$ as well as the right creation operators $\{r_i\}_{i=1}^d$ characterize rationality of operators generated by free semicircular elements $\bm{s}=(s_1,\ldots,s_d)$. 

Our main theorem can be stated as follows.
\begin{thm}[Theorem \ref{main_result}]
Let $a$ be in a von Neumann algebra $L^{\infty}(\bm{s})$ generated by $\bm{s}$. Then $\{[r^*_i,a]\}_{i = 1}^d$ are finite rank operators on $\mathcal{F}(H)$ if and only if  $a \in C_{\mathrm{div}}(\bm{s})$. In addition, we have
$$C_{\mathrm{div}}(\bm{s}) = C_{\mathrm{rat}}(\bm{s}) \subset \overline{\C \langle \bm{s} \rangle}$$
where $\overline{\C \langle \bm{s} \rangle}$ is the norm closure of noncommutative polynomials $\C \langle \bm{s} \rangle$ in $L^{\infty}(\bm{s})$.  
\end{thm}
In this theorem, we consider two notions of rationality, division closure $C_{\mathrm{div}}(\bm{s})$ and rational closure $C_{\mathrm{rat}}(\bm{s})$ of $\C \langle \bm{s} \rangle$ in $L^{\infty}(\bm{s})$ (see Definition \ref{rationality}). Roughly speaking, both closures describe bounded operators constructed from $\C \langle \bm{s} \rangle$ by combinations of algebraic operations $+,-,\times,\cdot^{-1}$.

Moreover, as a consequence of this theorem, we prove an analogue of Linnell's work which extends the above theorem to the algebra $\widetilde{L^{\infty}(\bm{s})}$ of closed densely defined (unbounded) linear operators affiliated with $L^{\infty}(\bm{s})$.
\begin{thm}[Theorem \ref{Linnell}]
Let $u \in \widetilde{L^{\infty}(\bm{s})}$ represented by $u = f^{-1} a = b g^{-1}$ for $a,b,f,g \in L^{\infty}(\bm{s})$. Then $\{f r_i^* b - a r_i^* g\}_{i=1}^d$ are finite rank operators if and only if $u$ belongs to the division closure $D(\bm{s})$ of $\C \langle \bm{s} \rangle$ in $\widetilde{L^{\infty}(\bm{s})}.$  
Moreover, we have $D(\bm{s}) \cap L^{\infty}(\bm{s}) = C_{\mathrm{div}}(\bm{s})$, and for any $u \in D(\bm{s})$, we can take $a,b,f,g$ so that they belong to $C_{\mathrm{div}}(\bm{s})$.   
\end{thm}  
In the last few years, rational quotients have been the object of studies in terms of free probability theory. In \cite{MSY19}, T. Mai, R. Speicher, and S. Yin found equivalent conditions for a tuple of operators so that one can evaluate any noncommutative rational functions in them. It is interesting to note that this condition is related to Voiculescu's (non-microstates) free entropy dimension and some quantity introduced by A. Connes and D. Shlyakhtenko \cite{CS05} in the context of their $L^2$-homology for von Neumann algebras. These studies have also been applied to random matrix models obtained from noncommutative rational functions in \cite{CMMPY21}. In \cite{ACSY21}, one can see how to estimate atoms of operators obtained from noncommutative rational functions evaluated in a free tuple of normal operators with the prescribed atoms.
 Our results show additional evidence that tools in free probability ought to play a natural role when we study rationality in a noncommutative setting.

This paper consists of four sections including the introduction. We explain the basics of the full Fock space and noncommutative rational series in Section \ref{Preliminaries}. Then we prove our main result in Section \ref{Proof of main theorem}.  In Section \ref{Rationality criterion for affiliated operators}, we extend this result to affiliated operators like in Linnell's work.

\section*{Acknowledgement}
This paper was completed during the Ph.D. studies of the author at Kyoto University under the supervision of B. Collins, who gave a talk about our paper \cite{CMMPY21} and received comments from A. Connes and R. Speicher on this occasion. Their feedback was the starting point of this paper. The author wants to thank them both for sharing their thoughts and references. 
The author also thanks T. Mai and S. Yin for reading a preliminary version of this manuscript and for helpful comments.
The author is grateful to an anonymous referee for useful comments and suggestions, which lead so substantial improvements in readability and clarity.

A. Miyagawa was supported by JST, the establishment of university fellowships towards the creation of science technology innovation, Grant Number JPMJFS2123 and by JSPS Research Fellowships for Young Scientists, JSPS KAKENHI Grant Number JP 22J12186.
      
\section{Preliminaries}\label{Preliminaries}
 \subsection{Full Fock space}
In the beginning of this section, we introduce free semicircular elements which are represented on the full Fock space. Let $d$ be a positive integer and $H$ be a complex Hilbert space with the dimension $d$ and an inner product $\langle \cdot , \cdot \rangle_{H}$. Then we consider the full Fock space as an orthogonal sum of Hilbert spaces $H^{\otimes n}$,
$$\mathcal{F}(H) = \bigoplus_{n=0}^{\infty} H^{\otimes n}$$
 where $H^{\otimes 0} = \C \Omega$ with a unit vector $\Omega$. $\langle \cdot, \cdot \rangle_{\mathcal{F}(H)}$ denotes the inner product on $\mathcal{F}(H)$. Note that $\langle \cdot, \cdot \rangle_{\mathcal{F}(H)}$ satisfies 
$$\langle \xi_1 \otimes \xi_2 \otimes \cdots \otimes \xi_m,  \eta_1 \otimes \eta_2 \otimes \cdots \otimes \eta_n \rangle_{\mathcal{F}(H)} = \delta^m_n \prod_{i=1}^m \langle \xi_i , \eta_i \rangle_H $$  
for $m,n \in \N$ and $\xi_i, \eta_i \in H$ where $\delta^m_n$ is the Kronecker's delta.
We consider a set of letters $[d] = \{1,\ldots,d\}$ and $[d]^*$ be the set of words which is the free monoid generated by $[d]$ with the empty word $\Omega$ (i.e. the identity in $[d]^*$). We denote by $|v|$ the length of a word $v \in [d]^*$. 

We associate words in $[d]^*$ with an orthonormal basis of $\mathcal{F}(H)$. Let $\{e_i\}_{i=1}^{d}$ be an orthonormal basis of $H$ and we define $e_v = e_{v_1} \otimes e_{v_2} \otimes \cdots \otimes e_{v_n}$ for $v=v_1v_2\cdots v_n \in [d]^*$ and $e_{\Omega} = \Omega$. Then $\{e_v\}_{v \in [d]^*}$ is an orthonormal basis of $\mathcal{F}(H)$. 

Let $B(\mathcal{F}(H))$ denote the set of bounded operators on $\mathcal{F}(H)$. Now we introduce our main objects. 
\begin{dfn}
For $f \in H$, we define the \emph{left creation operator} $l(f) \in B(\mathcal{F}(H))$ by 
$$l (f)e_v =  f \otimes e_v,$$
and we also call its adjoint operator $l(f)^* \in B(\mathcal{F}(H))$ the \emph{left annihilation operator}.  

In addition, we define the \emph{right creation operator} $r(f) \in B(\mathcal{F}(H))$ by
$$r(f) e_v = e_v \otimes f,$$
and the \emph{right annihilation operator} by $r(f)^*$. 
\end{dfn}
Note that $l(f)^*$ satisfies $l(f)^*e_v = \langle e_{v_1},f\rangle_{H} e_{v_2\cdots v_n}$ and $r(f)^* e_v = \langle e_{v_n} , f \rangle_H e_{v_1\cdots v_{n-1}}$ for $v=v_1v_2\cdots v_n \in [d]^*.$ 
Throughout this paper, we put $l_i = l(e_i)$ and $r_i = r(e_i)$ and $s_i = l_i + l_i^*$ for each $ i \in [d].$ 

Let $\C\langle \bm{s} \rangle$ denote the $\ast$-algebra of noncommutative polynomials in $\bm{s} =(s_1,\ldots,s_d)$; note that elements of $\C\langle \bm{s} \rangle$ can be written as 
$$\sum_{\substack{v \in [d]^* \\ |v|\le N}} \alpha_v s^v$$
where $N \in \Z_{\ge 0}, \ \alpha_v \in \C$ and $s^v = s_{v_1} \cdots s_{v_n}$ for any $v= v_1v_2\cdots v_n \in [d]^*.$
 
We define $L^{\infty}(\bm{s})$ as the von Neumann subalgebra of $B(\mathcal{F}(H))$ generated by $\bm{s}=(s_1,\ldots,s_d)$. In other words, $L^{\infty}(\bm{s})$ is the closure of $\C\langle \bm{s} \rangle$ in strong operator topology. 

We remark that $\Omega$ is cyclic and separating for $L^{\infty}(\bm{s})$, i.e. $\overline{L^{\infty}(\bm{s})\Omega} = \mathcal{F}(H)$, and if $X \in L^{\infty}(\bm{s})$ satisfies $X \Omega = 0$, then we have $X =0$.  

We define the vacuum state $\tau_{\Omega} $ on $L^{\infty}(\bm{s})$ by
$$\tau_{\Omega}(\cdot) = \langle \cdot \ \Omega, \Omega \rangle_{\mathcal{F}(H)}.$$
Then it is known that $\tau_{\Omega}$ is a tracial state for $L^{\infty}(\bm{s})$, i.e. $\tau_{\Omega}(X Y) = \tau_{\Omega}(YX)$ for $X,Y \in L^{\infty}(\bm{s})$.

Moreover, $\{s_i\}_{i=1}^d$ have free semicircle distributions with respect to $\tau_{\Omega}$. For these properties of $\Omega$ and $\tau_{\Omega}$, we refer to \cite[Theorem 2.6.2]{VDN92} or \cite[Corollary 7.17 and Proposition 7.18]{NS06}.

Let $L^2(\bm{s})$ be the Hilbert space obtained from $L^{\infty}(\bm{s})$ by completion with respect to the inner product $\langle x , y \rangle_2 = \tau_{\Omega}(y^* x), \ x,y \in L^{\infty}(\bm{s})$. Since the map 
\begin{eqnarray*}
L^{\infty}(\bm{s})  &\to& \mathcal{F}(H) \\
                              X & \mapsto & \hat{X} = X\Omega 
\end{eqnarray*}
is an isometry and $\Omega$ is cyclic (i.e. $\mathcal{F}(H)=\overline{L^{\infty}(\bm{s})\Omega}$), this map can be extend to a unitary operator from $L^2(\bm{s})$ to $\mathcal{F}(H).$     

Indeed, via the map $X \mapsto \hat{X}$, we can represent an orthonormal basis $\{e_v\}_{v \in [d]^*}$ of $\mathcal{F}(H)$ as elements in $L^{\infty}(\bm{s})$ by using the \emph{Chebyshev polynomials of the second kind}. Recall that the Chebyshev polynomials $U_n(X) \in \C[X]$ of the second kind are defined by the following recursion;
$$ U_{-1}(X) = 0,\ U_0(X) = 1,\ U_{n+1}(X) = X U_{n}(X) - U_{n-1}(X).$$  
For $w = i_1^{k_1}i_2^{k_2}\cdots i_n^{k_n}$ (where $i_1 \neq i_2 \neq \cdots \neq i_n$), we define an element $U_w$ in $L^{\infty}(\bm{s})$ by
$$U_{w} = U_{k_1}(s_{i_1})U_{k_2}(s_{i_2}) \cdots U_{k_n}(s_{i_n})$$
and also define $U_{\Omega} = 1$. Then $U_w$ is defined for any $w \in [d]^*$.

We can see by induction on the word length that for any $v \in [d]^*$, 
 $$\hat{U}_v = e_v.$$ 
This equality is also remarked in \cite[Section 5.1]{BS98}.
Since our main result focuses on noncommutative polynomials over $\bm{s} = (s_1,\ldots,s_d)$, we will use $\hat{U}_v$ rather than $e_v$. 

In order to represent annihilation operators $l_i^*, r_i^*$  in terms of word translations, we introduce the following operations which are also introduced in \cite[Chapter 1]{BR10}.  
\begin{dfn}
Let $0$ be a new letter. For $v \in [d]^* \sqcup \{0\}$ and $w \in [d]^*$, we define 
$$v w^{-1} = \begin{cases} v' \ \mathrm{if} \ v = v' w, \ v' \in [d]^* \\ 0 \ \mathrm{otherwise}\end{cases} $$
and also define
$$w^{-1} v = \begin{cases} v' \ \mathrm{if} \ v = w v', \ v' \in [d]^* \\ 0 \ \mathrm{otherwise.}\end{cases} $$
\end{dfn}
Put $U_0 = 0$. One should be careful that we use the same notation $U_0$ for $U_0=0$ and $U_0(X)=1$ ($U_0(X)$ corresponds with $U_{\Omega}$ in our definition). Then by using above notations we have for each $i \in [d]$
$$l^*_i(\hat{U}_v) = \hat{U}_{i^{-1} v} \quad \mathrm{and} \quad r^*_i(\hat{U}_v) = \hat{U}_{v i^{-1}}, \ v \in [d]^*.$$ 
\subsection{Noncommutative rational series}
In the sequel, we explain about noncommutative rational series.   
First, we give two definitions of rationality in a setting of unital algebras (over $\C$) as follows (see \cite[Definition 6]{BR10} or \cite[Definition 4.6 and 4.8]{MSY19}).
\begin{dfn}\label{rationality}
Let $\mathcal{A}$ be a unital algebra and $\mathcal{B} \subset \mathcal{A}$ be a unital subalgebra of $\mathcal{A}.$
We define the \emph{division closure} of $\mathcal{B}$ in $\mathcal{A}$ as the smallest unital subalgebra $\mathcal{C}$ of $\mathcal{A}$ such that $\mathcal{C}$ contains $\mathcal{B}$ and satisfies 
$$x \in \mathcal{C} \ \mathrm{is \ invertible \ in} \ \mathcal{A} \implies x^{-1} \in \mathcal{C}.$$
 In addtion, we define the \emph{rational closure} of $\mathcal{B}$ in $\mathcal{A}$ as the smallest (unital) subalgebra $D$ of $\mathcal{A}$ such that $\mathcal{D}$ contains $\mathcal{B}$ and satisfies for any $n \in \N,$
$$ X \in M_n(\mathcal{D}) \ \mathrm{is \ invertible \ in }\ M_n(\mathcal{A}) \implies X^{-1} \in M_n(\mathcal{D}).$$    
\end{dfn}
 Obviously, the division closure of any subalgebra is always contained in the rational closure of the same subalgebra, however, the converse is not necessarily true (Exercise 7.1.3 in \cite{cohn06}).

We will use facts for noncommutative rational series specific to our setting. The proofs of these results can be found in \cite[Chapter 1]{BR10}. 
We consider the algebra $\C \langle\langle [d] \rangle\rangle$ of \emph{noncommutative formal power series} with formal (noncommutative) variables $\{X_i\}_{i \in [d]}$ like as $\sum_{v \in [d]^*}\alpha_{v} X^v$ where $X^v = X_{v_1}X_{v_2}\cdots X_{v_n}$ for $v = v_1v_2\cdots v_n \in [d]^*$. Let $\C \langle [d] \rangle$ denote the subalgebra of noncommutative polynomials.

\begin{dfn}\label{recognizable}
Let $Z = \sum_{v \in [d]^*}\alpha_{v} X^v \in  \C \langle\langle [d] \rangle\rangle$. We say $Z$ is \emph{recognizable} if there exists $m \in \N$ and a \emph{linear representation} $(\lambda,\mu,\gamma)$ of dimension $m$ which consists of  a multiplicative map $\mu : [d]^* \to M_m(\C)$ (i.e. $\mu(v w) = \mu(v) \mu(w)$ for any $v,w \in [d]^*$) and $\lambda,\gamma \in \C^m$ such that for any $v \in [d]^*$
$$ \alpha_v = {}^t \lambda \mu(v) \gamma. $$
\end{dfn}
Let us say $Z$ is \emph{rational} if $Z$ belong to the division closure of $\C \langle [d] \rangle $ in $\C \langle\langle [d] \rangle\rangle$. Then the following theorem, known as the fundamental theorem, is crucial in this paper. This result is a collection of several works by Fliess, Jacobi, Kleene, and Sch\"{u}tzenberger. 
\begin{thm}[Corollary 1.5.4 and Theorem 1.7.1 in \cite{BR10}]\label{fundamental_theorem}
 Let $Z = \sum_{v \in [d]^*}\alpha_{v} X^v \in \C \langle\langle [d] \rangle\rangle$. Then the following are equivalent.
\begin{enumerate}
\item A $\C$-vector subspace of $\C \langle\langle [d] \rangle\rangle$ generated by $\sum_{v \in [d]^*}\alpha_{v} X^{v w^{-1}}$ ($w \in [d]^*$) is finitely generated. \\
\item A $\C$-vector subspace of $\C \langle\langle [d] \rangle\rangle$ generated by $\sum_{v \in [d]^*}\alpha_{v} X^{w^{-1}v}$ ($w \in [d]^*$) is finitely generated. \\
\item $Z$ is recognizable.\\
\item $Z$ is rational.
 \end{enumerate}
 \end{thm}
Moreover, if a noncommutative formal power series is recognizable and its linear representation $(\lambda, \mu, \gamma)$ has the minimal dimension, then $\mu$ is determined by its coefficients. This can be stated as follows.  
\begin{thm}[Corollary 2.2.3 in \cite{BR10}]\label{minimal_representation}
Suppose $Z = \sum_{v \in [d]^*}\alpha_{v} X^v \in \C \langle\langle [d] \rangle\rangle$ is recognizable with a linear representation $(\lambda,\mu,\gamma)$ which has the minimal dimension $m$. Then there exist $\{u_k\}_{k=1}^K,\{w_l\}_{l=1}^L \subset [d]^*$ and $c_{i j}^{k l} \in \C$ such that for any $v \in [d]^*$ and $1 \le i,j \le m$ 
$$\mu(v)_{i j} = \sum_{k l} c_{i j}^{k l} \alpha_{u_k v w_l}.$$ 
\end{thm}
In addition, we use an operation between noncommutative formal power series.
For $Z_1=\sum_{v \in [d]^*}\alpha_{v} X^v ,Z_2=\sum_{v \in [d]^*}\beta_{v} X^v \in \C \langle\langle [d] \rangle\rangle$, we define the \emph{Hadamard product} $Z_1 \odot Z_2$ by
$$Z_1 \odot Z_2 = \sum_{v \in [d]^*}\alpha_{v}\beta_{v} X^v.$$ 
One of the connections between the Hadamard product and rationality can be stated as follows.
\begin{thm}[Theorem 1.5.5 in \cite{BR10}]\label{Hadamard_product}
If $Z_1,Z_2 \in \C \langle\langle [d] \rangle\rangle$ are rational, then $Z_1 \odot Z_2$ is also rational.
\end{thm}

\subsection{Kronecker's theorem}\label{kronecker's theorem}

We need to recall Kronecker's theorem which basically tells us the equivalence between bounded rational functions and finite rank Hankel operators. 

Let $\{\alpha_n\}_{n=0}^{\infty}\subset \C$. We call a bounded operator $H$ on $l^2(\Z_{\ge 0})$ the \emph{Hankel operator} with respect to $ \{\alpha_n\}_{n=0}^{\infty}$ if $H$ satisfies
$$\langle H e_m,e_n \rangle = \alpha_{m+n}$$
for any $m,n \in \Z_{\ge 0}$ where $\{e_m\}_{m=0}^{\infty}$ is the standard orthonormal basis of $l^2(\Z_{\ge 0})$. The following theorem is known as Kronecker's theorem for the studies of Hankel operators (see \cite{Kro1881} and \cite[Theorem 3.11]{Partington88}).
\begin{thm}\label{Kronecker}
Let $\{\alpha_n\}_{n=0}^{\infty} \subset \C$. Then a formal Laurent series (in $z^{-1}$) $a(z) = \sum_{n=0}^{\infty}\alpha_n z^{-n-1}$ is a rational function (i.e. $a(z)= \frac{P(z)}{Q(z)}$ for some polynomials $P(z),Q(z)$) such that all poles of $a(z)$ are contained in $\{z \in \C \ | \ |z| <1\}$ 
if and only if $\{\alpha_n\}_{n=0}^{\infty}$ determines a finite rank Hankel operator. In this case, the number of poles on $f$ is equal to the rank of the Hankel operator.    
\end{thm}

Here, we explain a related recursion and estimate in Theorem \ref{Kronecker} in order to explain Corollary \ref{estimation of coefficients}, which we will use in the proof of Corollary \ref{convergence}.  
Indeed, if $a(z) = \sum_{n=0}^{\infty}\alpha_n z^{-n-1}$ is rational and the denominator of $a(z)$ written as $Q(z)=\sum_{k=0}^m \lambda_k z^k$ ($\lambda_m \neq 0$), 
then we have the following recursion for $\{\alpha_n\}_{n=0}^{\infty}$
$$\sum_{k =0}^m \lambda_k \alpha_{n+k} = 0,$$
where $\{\alpha_n\}_{n=0}^{m-1}$ are determined by the numerator of $a(z)$. 
This recursion is characterized by the poles of $a(z)$, and if we additionally assume $\lim_{n \to \infty}\alpha_n =0$, we can see that all poles of $a(z)$ are contained in $\{z \in \C \ | \ |z| <1\}$ (see the proof of \cite[Theorem 3.11]{Partington88}).  
Moreover, this implies $|\alpha_n|$ is bounded above by $M c^n$ where $M > 0$ and $c = \max \{|p| \ | \ p \mathrm{\ is \ a \ pole \ of} \ a(z)\}.$ 

By replacing $a(z)$ by $z a(z^{-1})$, we obtain the following estimate from the above observation, which is used in the proof of \cite[Lemma 10]{DR97}.         
\begin{cor}\label{estimation of coefficients}
Let $a(z)= \sum_{n=0}^{\infty} \alpha_n z^n$ be a formal power series with $\sum_{n=0}^{\infty} |\alpha_n|^2 < \infty$. If $a(z)$ is rational, then there exists $M >0$ and $0<c<1$ such that we have for any $n \in \N$ 
$$  |\alpha_n| \le M c^n. $$
\end{cor}

\section{Main results}\label{Proof of main theorem}

Let $C_{\mathrm{div}}(\bm{s})$ denote the division closure of $\C \langle \bm{s} \rangle$ in $L^{\infty}(\bm{s})$ and $C_{\mathrm{rat}}(\bm{s})$ denote the rational closure of  $\C \langle \bm{s} \rangle$ in $L^{\infty}(\bm{s}).$ 
 
Let us state our main theorem again. 
\begin{thm}\label{main_result}
Let $a \in L^{\infty}(\bm{s})$. Then $\{[r^*_i,a]\}_{i = 1}^d$ are finite rank operators on $\mathcal{F}(H)$ if and only if  $a \in C_{\mathrm{div}}(\bm{s})$. In addition, we have
$$C_{\mathrm{div}}(\bm{s}) = C_{\mathrm{rat}}(\bm{s}) \subset \overline{\C \langle \bm{s} \rangle}$$
where $\overline{\C \langle \bm{s} \rangle}$ is the norm closure of $\C \langle \bm{s} \rangle$ in $L^{\infty}(\bm{s})$.  
\end{thm}
We basically follow the proof by G. Duchamp and C. Reutenauer \cite{DR97}.
The following two lemmas have important roles in proving our main theorem.
  
\begin{lm}\label{dual_system}
For any $i,j \in [d]$ and $k \in \N$, we have
$$[r_i^*, U_k(s_j)] = \delta^i_j \sum_{l = 1}^k U_{l-1}(s_j)P_{\Omega} U_{k-l}(s_j)$$
where $P_{\Omega}$ is the orthogonal projection onto $\hat{U}_{\Omega} = \Omega$.
\end{lm}
\begin{proof}
This lemma is easily deduced from the property of Chebyshev polynomials (see \cite[Exercise 10 in Section 8.8]{MS17}) and a dual system (see \cite[Semicircular Example 5.13]{Voi98}). However, we give a proof of this lemma for the purpose of self-containment.

First, we show for any $i,j \in [d]$
$$[r_i^*,s_j] = \delta^i_j P_{\Omega}.$$
For any $v \in [d]^*$ we have
\begin{eqnarray*}
[r_i^*, s_j ] \hat{U}_v&=& [r_i^*,l_j^* + l_j] \hat{U}_v \\ 
 &=&r_i^*(l_j^* + l_j) \hat{U}_v - (l_j^* + l_j) r_i^* \hat{U}_v \\
&=&\hat{U}_{(j^{-1}v) i^{-1}} + \hat{U}_{(j v) i^{-1}} - \hat{U}_{j^{-1}(v i^{-1})} - \hat{U}_{j(v i^{-1})}. 
\end{eqnarray*}
Note that $[r_i^*, s_j ] \hat{U}_v = 0$ except for $v = \Omega$ and in this case we have 
$$[r_i^*, s_j ] \hat{U}_{\Omega} = \hat{U}_{j i^{-1}} = \delta^i_j \hat{U}_{\Omega}.$$

Then we can compute $[r_i^*, U_k(s_j)]$ by induction since by the Leibniz rule we have  
\begin{eqnarray*}
[r_i^*,U_{k+1}(s_j)]&=& [r_i^*,s_j U_{k}(s_j)] - [r_i^*,U_{k-1}(s_j)] \\
&=&[r_i^*,s_j ]U_k(s_j) + s_j [r_i^*,U_k(s_j)]  - [r_i^*,U_{k-1}(s_j)] \\
&=& \delta^i_j P_{\Omega} U_k(s_j) + s_j [r_i^*,U_k(s_j)]  - [r_i^*,U_{k-1}(s_j)]
\end{eqnarray*}
and also have by the recursion formula of $U_k$ 
\begin{eqnarray*}
s_j \sum_{l = 1}^k U_{l-1}(s_j)P_{\Omega} U_{k-l}(s_j) &=&\sum_{l = 1}^k U_{l}(s_j)P_{\Omega} U_{k-l}(s_j) + \sum_{l = 2}^k U_{l-2}(s_j)P_{\Omega} U_{k-l}(s_j) \\
&=& \sum_{l = 2}^{k+1} U_{l-1}(s_j)P_{\Omega} U_{k+1-l}(s_j) + \sum_{l = 1}^{k-1} U_{l-1}(s_j)P_{\Omega} U_{k-1-l}(s_j).
\end{eqnarray*}
By multiplying by $\delta^i_{j}$ and using the induction hypothesis,
$$ s_j [r_i^*,U_k(s_j)]=\delta^i_j \sum_{l = 2}^{k+1} U_{l-1}(s_j)P_{\Omega} U_{k+1-l}(s_j) + [r_i^*,U_{k-1}(s_j)],$$
which gives the asserted formula for $[r_i^*,U_{k+1}(s_j)].$
\end{proof}
\begin{lm}\label{fundamental_equality_free}
For $v,w \in [d]^*$ and $i \in [d]$, we have
$$[r^*_i, U_v]\hat{U}_w = \hat{U}_{v (i w^*)^{-1}}$$
where $w^*$ is the transpose of $w$, in other words $w^* = i_n^{k_n}i_{n-1}^{k_{n-1}} \cdots i_1^{k_1}$ when $w = i_1^{k_1}i_2^{k_2}\cdots i_n^{k_n}$.  
\end{lm}    
\begin{proof}
By Lemma \ref{dual_system} and the fact that $[r_i^*,\cdot]$ is a derivation, we have for $v=i_1^{k_1}i_2^{k_2}\cdots i_n^{k_n},$
\begin{eqnarray*}
&\ & [r_i^*, U_v] \hat{U}_w \\ 
&=& \left(\sum_{m=1}^n U_{k_1}(s_{i_1}) \cdots U_{k_{m-1}}(s_{i_{m-1}})[r_i^*, U_{k_m}(s_{i_m})]U_{k_{m+1}}(s_{i_{m+1}}) \cdots U_{k_n}(s_{i_n})  \right) \hat{U}_w \\
&=&\sum_{m=1}^n\sum_{j=1}^{k_m} \delta_i^{i_m} U_{k_1}(s_{i_1}) \cdots U_{k_{m-1}}(s_{i_{m-1}})U_{j-1}(s_{i_m})P_{\Omega} U_{k_m - j}(s_{i_m})U_{k_{m+1}}(s_{i_{m+1}}) \cdots U_{k_n}(s_{i_n}) \hat{U}_w \\
&=& \sum_{m=1}^n\sum_{j=1}^{k_m} \delta_i^{i_m} U_{i_1^{k_1}\cdots i_m^{j-1}}P_{\Omega}U_{i_m^{k_m -j}\cdots i_n^{k_n}} \hat{U}_w.
\end{eqnarray*}
Since we have 
\begin{eqnarray*}P_{\Omega}U_{i_m^{k_m -j}\cdots i_n^{k_n}} \hat{U}_w &=& \langle U_{i_m^{k_m -j}\cdots i_n^{k_n}} \hat{U}_w, \Omega \rangle_{\mathcal{F}(H)} \Omega \\
&=&\langle\hat{U}_w , \hat{U}_{i_n^{k_n}\cdots i_m^{k_m -j}} \rangle_{\mathcal{F}(H)}\Omega
\end{eqnarray*}
 and $\{\hat{U}_w\}_{w \in [d]^*}$ is an orthonormal basis, we conclude
\begin{eqnarray*}
[r^*_i,U_v] \hat{U}_w &=&  \sum_{m=1}^n\sum_{j=1}^{k_m} \delta_i^{i_m}\langle\hat{U}_w , \hat{U}_{i_n^{k_n}\cdots i_m^{k_m -j}}\rangle_{\mathcal{F}(H)} \hat{U}_{i_1^{k_1}\cdots i_m^{j-1}} \\
&=& \hat{U}_{v(i w^*)^{-1}}.
\end{eqnarray*}
\end{proof}
Next we associate elements $\sum_{v \in [d]^*}\alpha_{v} \hat{U}_v$ in $\mathcal{F}(H)$ with  noncommutative formal power series $\sum_{v \in [d]^*} \alpha_v X^{v}$. Since $U_v U_w \neq U_{vw}$ in general, we cannot directly connect $U_v$ with $X^v$ while keeping a multiplicative structure. However we can connect them by using a matrix representation, which may help us to prove our main theorem.
\begin{lm}\label{correspondence}
For each $i \in [d]$, we put 
$$S_i = E_{i i} \otimes \begin{pmatrix}s_i & -1 \\ 1 & 0\end{pmatrix} + \sum_{j \neq i } E_{j i} \otimes  \begin{pmatrix}s_i & -1 \\ 0 & 0\end{pmatrix} \in M_d(\C) \otimes M_2(L^{\infty}(\bm{s}))$$
where $E_{j i} \in M_d(\C)$ is a matrix whose $(j,i)$ entry is $1$ and other entries are $0$.
Then for $v =  i_1^{k_1} i_2^{k_2} \cdots i_n^{k_n} \in [d]^*(i_1 \neq i_2 \neq \cdots \neq i_n)$ we have
$$U_v =  \begin{pmatrix} 1 & 0 \end{pmatrix} ({}^t e_1 \otimes I_2) \ S_{i_1}^{k_1} S_{i_2}^{k_2} \cdots S_{i_n}^{k_n} \ (e \otimes I_2) \begin{pmatrix} 1 \\ 0 \end{pmatrix}$$
where $I_2$ is the identity matrix and $\{e_i\}_{i \in [d]} \subset \C^d$ is the standard basis of $\C^d$, and we put $e = \sum_{i=1}^d e_i$.      
\end{lm}
 
\begin{proof}
Since Chebyshev polynomials $U_n(X)$ satisfy for $n \in \N$
  $$\begin{pmatrix}U_n(X) \\ U_{n-1}(X) \end{pmatrix} = \begin{pmatrix} X & -1 \\ 1 & 0\end{pmatrix}\begin{pmatrix} U_{n-1}(X) \\ U_{n-2}(X)\end{pmatrix},$$
we can show that
$$ \begin{pmatrix} X & -1 \\ 1 & 0\end{pmatrix}^n =  \begin{pmatrix}U_n(X) & -U_{n-1}(X) \\ U_{n-1}(X)& -U_{n-2}(X) \end{pmatrix}.$$
In particular, we have for any $i \in [d]$ and $n \in \Z_{\ge 0}$
$$U_{i^n} =  \begin{pmatrix} 1 & 0 \end{pmatrix}  \begin{pmatrix}s_i & -1 \\ 1 & 0\end{pmatrix}^n \begin{pmatrix} 1 \\ 0 \end{pmatrix}.$$
Then we have for $v = i_1^{k_1} i_2^{k_2} \cdots i_n^{k_n}$ 
\begin{eqnarray*}
U_v &=& U_{i_1^{k_1}} U_{i_2^{k_2}} \cdots U_{i_n^{k_n}} \\
&=& \prod_{l = 1}^n \begin{pmatrix} 1 & 0 \end{pmatrix}\begin{pmatrix} s_{i_l} & -1 \\ 1 & 0 \end{pmatrix}^{k_l} \begin{pmatrix} 1 \\ 0 \end{pmatrix}\\
&=& \begin{pmatrix} 1 & 0 \end{pmatrix} \left[\prod_{l = 1}^n P\begin{pmatrix} s_{i_l} & -1 \\ 1 & 0 \end{pmatrix}^{k_l} P \right] \begin{pmatrix} 1 \\ 0 \end{pmatrix}
\end{eqnarray*}
where we put $P =\begin{pmatrix} 1 & 0 \\ 0 & 0 \end{pmatrix}$. Note that $P \begin{pmatrix}s_i & -1 \\ 1 & 0\end{pmatrix} = \begin{pmatrix}s_i & -1 \\ 0 & 0\end{pmatrix}$. Since we have 
$$S_i^n = E_{i i} \otimes \begin{pmatrix}s_i & -1 \\ 1 & 0\end{pmatrix}^n + \sum_{j \neq i } E_{j i} \otimes  P\begin{pmatrix}s_i & -1 \\ 1 & 0\end{pmatrix}^n,$$
we obtain for $i_1 \neq i_2 \neq \cdots \neq i_n$
\begin{eqnarray*}
 S_{i_1}^{k_1} S_{i_2}^{k_2} \cdots S_{i_n}^{k_n} &=& \prod_{l=1}^n \left[ E_{i_l i_l} \otimes \begin{pmatrix}s_{i_l} & -1 \\ 1 & 0\end{pmatrix}^{k_l} + \sum_{j \neq i_l } E_{j i_l} \otimes  P\begin{pmatrix}s_{i_l} & -1 \\ 1 & 0\end{pmatrix}^{k_l}\right] \\
&=& E_{i_1 i_n} \otimes \left[\begin{pmatrix}s_{i_1} & -1 \\ 1 & 0\end{pmatrix}^{k_1}\prod_{l = 2}^n P\begin{pmatrix} s_{i_l} & -1 \\ 1 & 0 \end{pmatrix}^{k_l}\right]  \\
&\ & \qquad + \sum_{j \neq i_1} E_{j i_n} \otimes \left[\prod_{l = 1}^n P\begin{pmatrix} s_{i_l} & -1 \\ 1 & 0 \end{pmatrix}^{k_l}\right]. 
\end{eqnarray*}
 Thus we conclude (note that $\begin{pmatrix} 1 & 0 \end{pmatrix} P = \begin{pmatrix} 1 & 0 \end{pmatrix}$),
\begin{eqnarray*}
\begin{pmatrix} 1 & 0 \end{pmatrix} ({}^t e_1 \otimes I_2) \ S_{i_1}^{k_1} S_{i_2}^{k_2} \cdots S_{i_n}^{k_n} \ (e \otimes I_2) \begin{pmatrix} 1 \\ 0 \end{pmatrix} &=& \begin{pmatrix} 1 & 0 \end{pmatrix} \left[\prod_{l = 1}^n P\begin{pmatrix} s_{i_l} & -1 \\ 1 & 0 \end{pmatrix}^{k_l} P \right] \begin{pmatrix} 1 \\ 0 \end{pmatrix} \\
&=& U_v.
\end{eqnarray*}
\end{proof}
Next step is to show convergence of $\sum_{w \in [d]^*}\alpha_{v} S^v$ under certain assumptions. In order to estimate the operator norm of $\sum_{w \in [d]^*}\alpha_{v} S^v$, we use the Haagerup type inequality for the full Fock space which was proved by M. Boz\.{e}jko \cite{Bozejko99} in terms of the  $q$-deformed Fock space. One can get the inequality for the full Fock space as $q=0$. Here, we revisit a proof of this inequality in the case $q=0$ for reader's convenience.

\begin{lm}[Haagerup inequality]\label{haagerup}
Let  $m \in \Z_{\ge 0}$ and $\{\alpha_v\}_{|v|=m}$ be a family of complex numbers. Then we have 
$$\left \|\sum_{|v| = m} \alpha_v U_v \right \| \le (m+1)\left \|\sum_{|v| = m} \alpha_v \hat{U}_v\right \|_{\mathcal{F}(H)},$$
where $\|\cdot\|$ is the operator norm on $B(\mathcal{F}(H))$ and $\| \cdot \|_{\mathcal{F}(H)}$ is the norm defined by $\sqrt{\langle \xi , \xi \rangle_{\mathcal{F}(H)}}$ for $\xi \in \mathcal{F}(H)$. 
\end{lm}
\begin{proof}
First, we show $$\max\left\{\left \|\sum_{|v|=m}\alpha_{v} l_v \right \|, \left \|\sum_{|v| = m}\alpha_{v} l^*_v\right \|\right\} \le \left \|\sum_{|v| = m} \alpha_v \hat{U}_v\right \|_{\mathcal{F}(H)}$$
where $l_v = l_{v_1}l_{v_2}\cdots l_{v_m}, \ l^*_v = l^*_{v_1}l^*_{v_2}\cdots l^*_{v_m}$ for $v = v_1v_2 \cdots v_m.$ Since we have for $\xi \in H^{\otimes  n}$
\begin{eqnarray*}
\left \|\sum_{|v| = m}\alpha_{v} l_v\xi \right \|_{\mathcal{F}(H)}^2 &=& \left \|\sum_{|v| = m}\alpha_{v} (e_v \otimes \xi)\right \|_{\mathcal{F}(H)}^2 \\
 &=& \sum_{|v| = m}|\alpha_{v}|^2 \| \xi \|_{\mathcal{F}(H)}^2 
\end{eqnarray*}
and $\sum_{|v| = m}\alpha_{v} l_v(\xi)$ and $\sum_{|v| = m}\alpha_{v} l_v(\eta)$ are orthogonal for $\xi \in H^{\otimes n}$ and $\eta \in H^{\otimes n'}, \ n \neq n'$, we have $\|\sum_{|v| = m}\alpha_{v} l_v\| \le \|\sum_{|v| = m} \alpha_v \hat{U}_v\|_{\mathcal{F}(H)}.$ Moreover, by taking involution, we have
\begin{eqnarray*} 
\left \|\sum_{|v| = m} \alpha_v l^*_v\right \| &=& \left \|\left(\sum_{|v| = m} \alpha_v l^*_v\right)^* \right\| \\
&=& \left \|\sum_{|v| = m} \overline{\alpha_v} l_{v^*}\right\| \\
&\le&\sqrt{\sum_{|v| = m} |\alpha_{v^*}|^2} \\
&=&  \left \|\sum_{|v| = m} \alpha_v \hat{U}_v\right \|_{\mathcal{F}(H)}.  
\end{eqnarray*}
In order to prove this lemma, we use the following characterization of $U_v$ for $v =  v_1 \cdots v_m \in [d]^*, v_i \in [d]$ (see Proposition 2.7 in \cite{BKS97})
$$U_v = \sum_{k=0}^m l_{v_1}\cdots l_{v_k}l^*_{v_{k+1}}\cdots l^*_{v_m}.$$ 
From this formula, we rewrite $\sum_{|v| = m} \alpha_v U_v$ by $\sum_{k=0}^{m} F^{(k)}$ where $F^{(k)}$ denotes
$$\sum_{\substack{|u|=k \\ |v| = m-k}} \alpha_{u v} l_u l^*_v$$
for $k = 0,\ldots,n$. We will show $\|F^{(k)}\| \le \|\sum_{|v| = m} \alpha_v \hat{U}_v\|_{\mathcal{F}(H)}$ for any $k$. Since we have already proved this for $k=0,m$ in the previous argument, we fix $k=1,\ldots,n-1$. In addition, since $F^{(k)}(\xi)$ and $F^{(k)}(\eta)$ are orthogonal when $\xi \in H^{\otimes n}$, $\eta \in H^{\otimes n'}$ where $n \neq n'$, it suffices to show that $\|F^{(k)}(\xi)\|_{\mathcal{F}(H)} \le \|\sum_{|v| = m} \alpha_v \hat{U}_v\|_{\mathcal{F}(H)}\|\xi\|_2$ for $\xi \in H^{\otimes n}$ where $n \ge m -k$ (note that $F^{(k)}(\xi) = 0$ when $n < m -k$).       
Then we have  
\begin{eqnarray*}
\|F^{(k)}\xi\|_{\mathcal{F}(H)}^2 &=& \langle \sum_{\substack{|u_1| = k \\ |u_2| = n-k}}  \alpha_{u_1 u_2} l_{u_1}l^*_{u_2} \xi,  \sum_{\substack{|v_1| = k \\ |v_2| = n-k}}  \alpha_{v_1 v_2} l_{v_1}l^*_{v_2}\xi   \rangle_{\mathcal{F}(H)} \\
&=& \sum_{\substack{|u_1|=|v_1| = k \\ |u_2|=|v_2| = n-k}} \alpha_{u_1 u_2} \overline{\alpha_{v_1 v_2}} \langle  l_{u_1}l^*_{u_2} \xi, l_{v_1}l^*_{v_2}\xi  \rangle_{\mathcal{F}(H)} \\
&=&  \sum_{\substack{|u_1|=|v_1| = k \\ |u_2|=|v_2| = n-k}} \alpha_{u_1 u_2} \overline{\alpha_{v_1 v_2}} \langle  e_{u_1},e_{v_1}\rangle_{\mathcal{F}(H)} \langle l^*_{u_2} \xi, l^*_{v_2}\xi  \rangle_{\mathcal{F}(H)}.
\end{eqnarray*}
Since $\{e_v\}_{v \in [d]^*}$ is an orthonormal basis of $\mathcal{F}(H)$, the last term is equal to 
\begin{eqnarray*}
\sum_{\substack{|u| = k \\ |u_2|=|v_2| = n-k}} \alpha_{u u_2} \overline{\alpha_{u v_2}} \langle l^*_{u_2} \xi, l^*_{v_2}\xi  \rangle_{\mathcal{F}(H)} &=& \sum_{|u| = k} \langle \sum_{|u_2| = n- k} \alpha_{u u_2} l^*_{u_2} \xi,\sum_{|v_2| = n- k}\alpha_{u v_2} l^*_{v_2}\xi  \rangle_{\mathcal{F}(H)} \\
&=& \sum_{|u| = k} \left \| \sum_{|v| = n- k} \alpha_{u v} l^*_{v} \xi \right \|^2_{\mathcal{F}(H)}. 
\end{eqnarray*}
Since we have $\left \|\sum_{|v| = n- k} \alpha_{u v} l^*_{v}\right \| \le \sqrt{\sum_{|v| = n-k}|\alpha_{u v}|^2}$, we obtain
$$  \|F^{(k)}\xi\|_2^2 \le \sum_{|u| = k}\sum_{|v| = n-k}|\alpha_{u v}|^2\|\xi\|_{\mathcal{F}(H)}^2 = \left \|\sum_{|v| = m} \alpha_v \hat{U}_v\right \|_{\mathcal{F}(H)}^2  \|\xi\|_{\mathcal{F}(H)}^2.$$
Thus we conclude
\begin{eqnarray*}
\left\|\sum_{|v| = m} \alpha_v U_v\right \| &=& \left \|\sum_{k=0}^m F^{(k)}\right \| \\
&\le& \sum_{k=0}^m  \|F^{(k)}\| \\
&\le& (m+1)\left \|\sum_{|v| = m} \alpha_v \hat{U}_v\right\|_{\mathcal{F}(H)}.
\end{eqnarray*}
\end{proof}

\begin{lm}\label{haagerup-application}
Let us take $S_1,\ldots,S_d \in M_d(\C) \otimes M_2(L^{\infty}(\bm{s}))$ as in Lemma \ref{correspondence}. Then we have for any $m \in \Z_{\ge 0}$
$$\left \|\sum_{|v| = m} \alpha_v S^v \right \| \le 4d^2 (m+1)\left \|\sum_{|v| = m} \alpha_v \hat{U}_v \right \|_{\mathcal{F}(H)}.$$   
\end{lm} 
\begin{proof}
When $m =0,1$, one can easily derive the above inequality from Lemma \ref{correspondence} and Lemma \ref{haagerup}. Thus we may suppose from now on that $m \ge 2$. Recall from the proof of Lemma \ref{correspondence} that we have for $v = i_1^{k_1} i_2^{k_2} \cdots i_n^{k_n}$  
$$S^v  =E_{i_1 i_n} \otimes \left[\begin{pmatrix}s_{i_1} & -1 \\ 1 & 0\end{pmatrix}^{k_1}\prod_{l = 2}^n P\begin{pmatrix} s_{i_l} & -1 \\ 1 & 0 \end{pmatrix}^{k_l}\right]  + \sum_{j \neq i_1} E_{j i_n} \otimes \left[\prod_{l = 1}^n P\begin{pmatrix} s_{i_l} & -1 \\ 1 & 0 \end{pmatrix}^{k_l}\right].$$
Since we have
\begin{eqnarray*}
\begin{pmatrix}s_{i_1} & -1 \\ 1 & 0\end{pmatrix}^{k_1}\prod_{l = 2}^n P\begin{pmatrix} s_{i_l} & -1 \\ 1 & 0 \end{pmatrix}^{k_l} &=& \begin{pmatrix}U_{i_1^{k_1}} & 0 \\ U_{i_1^{k_1-1}} & 0\end{pmatrix}\left(\prod_{l = 2}^{n-1} U_{i_l^{k_l}}\right) \begin{pmatrix} U_{i_n^{k_n}} & -U_{i_n^{k_n -1}} \\ 0 & 0 \end{pmatrix} \\
&=& \begin{pmatrix}U_v & - U_{vi_n^{-1}} \\ U_{i_1^{-1}v} & - U_{i_1^{-1}v i_n^{-1}} \end{pmatrix}, 
\end{eqnarray*}
$S^v$ can be written as the following form,
$$
S^v = E_{i_1 i_n} \otimes  \begin{pmatrix}U_v & - U_{vi_n^{-1}} \\ U_{i_1^{-1}v} & - U_{i_1^{-1}v i_n^{-1}} \end{pmatrix} + \sum_{j \neq i_1} E_{j i_n} \otimes \begin{pmatrix}U_v & - U_{vi_n^{-1}} \\ 0 & 0 \end{pmatrix}.
$$
Thus we have
\begin{eqnarray*}
&\ &\sum_{|v| = m} \alpha_v S^v = \\
 &\ & \sum_{i, j \in [d]} \sum_{|v| = m -2} \alpha_{ivj} \left( E_{i j} \otimes  \begin{pmatrix}U_{ivj} & - U_{iv} \\ U_{vj} & - U_{v} \end{pmatrix} + \sum_{k \neq i} E_{k j} \otimes \begin{pmatrix}U_{ivj} & - U_{iv} \\ 0 & 0 \end{pmatrix}\right) \\
&=& \sum_{i, j \in [d]} E_{i j} \otimes \begin{pmatrix}\sum_{k \in [d]}\sum_{|v| = m -2} \alpha_{kvj}U_{kvj} & - \sum_{k \in [d]}\sum_{|v| = m -2} \alpha_{kvj}U_{kv} \\ \sum_{|v| = m -2} \alpha_{ivj}U_{vj} & -\sum_{|v| = m -2} \alpha_{ivj} U_{v} \end{pmatrix} \\
 &=& \sum_{i, j \in [d]} E_{i j} \otimes \begin{pmatrix}\sum_{|v| = m -1} \alpha_{vj}U_{vj} & - \sum_{|v| = m -1} \alpha_{vj}U_{v} \\ \sum_{|v| = m -2} \alpha_{ivj}U_{vj} & -\sum_{|v| = m -2} \alpha_{ivj} U_{v} \end{pmatrix}.
\end{eqnarray*}
 Note that all entries of $\sum_{|v|=m}\alpha_v S^v$ are sums of $U_v$ ($|v| = m, \ m-1, \ m-2$) whose coefficients are subsequences of $\{\alpha_v\}_{|v| =m}$. Therefore by Lemma \ref{haagerup}, operator norms of all entries of $\sum_{|v|=m}\alpha_v S^v$ are bounded by $(m+1)\|\sum_{|v| = m} \alpha_v \hat{U}_v\|_{\mathcal{F}(H)}$ and we obtain a desired estimate by the triangle inequality.    
\end{proof}
By the same argument in the Lemma 10 of \cite{DR97}, we have the following corollary. 
\begin{cor}\label{convergence}
Let $\{\alpha_v\}_{v \in [d]^*}$ be a family of complex numbers such that $\sum_{v \in [d]^*}|\alpha_v|^2 < \infty$ and $\sum_{v \in [d]^*} \alpha_v X^v $ is rational as a noncommutative formal power series. We put $a_m = \sum_{|v| = m} \alpha_v S^v \in M_d(\C) \otimes M_2(L^{\infty}(\bm{s}))$. Then $\sum_{m = 0}^{\infty} a_m$ converges in the operator norm.   
\end{cor}
\begin{proof}
Note that $\sum_{v \in [d]^*} \overline{\alpha_v} X^v$ is also rational (i.e. recognizable) by taking a complex conjugate of each entry of a linear representation of the recognizable series $\sum_{v \in [d]^*} \alpha_v X^v$. Since the Hadamard product of two rational series is also rational by Lemma \ref{Hadamard_product}, $\sum_{v \in [d]^*} |\alpha_v|^2 X^v$ is also rational as a noncommutative formal power series. By evaluating $X_1,X_2\ldots,X_d$ in one variable $z$ (i.e. $X_1=X_2=\cdots=X_d=z$), we can use the argument of Kronecker (see Corollary \ref{estimation of coefficients}) for the formal power series 
$$\sum_{m=0}^{\infty} \left(\sum_{|v| = m} |\alpha_v|^2 \right) z^m.$$ 
Thus there exists $M>0$ and $0<c<1$ such that
$$\sum_{|v| = m}|\alpha_v|^2 \le M c^m.$$
By using Lemma \ref{haagerup}, we can estimate the operator norm of $a_m$ as
$$\|a_m\| \le 4d^2 (m+1)\sqrt{\sum_{|v|=m} |\alpha_v|^2} \le M' (m+1) c'^{m}$$
for some constant $M' >0$ and $0<c'<1$. Thus $\sum_{m = 0}^{\infty} a_m$ converges in operator norm.              
\end{proof}
We also use the following technical lemma. 
\begin{lm}[Lemma 11 in \cite{DR97}]\label{rational_entry}
Let $n \in \N$ and $\mathcal{A}$ be a Banach algebra. If $x \in M_n(\mathcal{A})$ satisfies $\lim_{m \to \infty} \|x^m\| = 0$, then we have
\begin{enumerate}
\item $\sum_{m =0}^{\infty} x^m$ converges in the operator norm to $(1-x)^{-1} \in M_n(\mathcal{A})$.
\item All entries of $(1-x)^{-1}$ belong to the division closure of the subalgebra generated by all entries of $x$ in $\mathcal{A}$.
\end{enumerate} 
\end{lm}

\begin{pr}\label{onlyif}
Let $a  \in  L^{\infty}(\bm{s}) $. If $\{[r^*_i,a]\}_{i = 1}^d$ are finite rank operators on $\mathcal{F}(H)$, then $a \in C_{\mathrm{div}}(\bm{s})$.   
\end{pr}
\begin{proof}
 Let $\hat{a} = \sum_{v \in [d]^*} \alpha_v \hat{U}_v$ be the expansion of $\hat{a}$ and $M$ be a $\C$-submodule of $\C\langle\langle [d]\rangle \rangle$ generated by $\sum_{v \in [d]^*}\alpha_{v} X^{v w^{-1}}$ ($w \in [d]^*$). Thanks to Lemma \ref{fundamental_equality_free}, we have for each $i \in [d]$ 
$$[r_i^*,a] \hat{U}_{w} = \sum_{v \in [d]^*} \alpha_v \hat{U}_{v(iw^*)^{-1}}.$$
Note that the linear map from $\mathcal{F}(H)$ to $\C\langle\langle [d]\rangle \rangle$ which maps $\hat{U}_v$ to $X^v$ for each $v \in [d]^*$ is injective. Therefore $M$ is finitely generated if $\{[r^*_i,a]\}_{i = 1}^d$ are finite rank operators. Thus the noncommutative formal power series $\sum_{v \in [d]^*}\alpha_{v} X^v$ is a recognizable series by Theorem \ref{fundamental_theorem}. 
In other words, there exists a linear representation which consists of a multiplicative morphism $\mu: [d]^* \to M_m(\C)$ and vectors $\lambda,\gamma \in \C^m$ such that $\alpha_v = {}^t \lambda \mu(v) \gamma$. Moreover by choosing a linear representation such that its dimension is minimal, we may assume from Theorem \ref{minimal_representation} that there exists $\{u_k\}_{k=1}^K,\{w_l\}_{l=1}^L \subset [d]^*$ such that $$\mu(v)_{i j} = \sum_{k l} c_{i j}^{k l} \alpha_{u_k v w_l}$$ for any $v \in [d]^*$ and $1 \le i , j \le m$.  
We put $V(\bm{X}) = \sum_{i \in [d]} \mu(i) X_i$. Note that since $\mu$ is multiplicative, $V(\bm{X})$ satisfies 
$$V(\bm{X})^m = \sum_{|v| = m} \mu(v) X^v$$
and, on the level of formal power series, we have 
\begin{eqnarray*}
\sum_{v \in [d]^*}\alpha_{v} X^v &=& {}^t \lambda \left[\sum_{m =0}^{\infty}V(\bm{X})^m\right] \gamma \\
                                           &=& {}^t \lambda [1-V(\bm{X})]^{-1} \gamma.
\end{eqnarray*}
 We evaluate $\bm{X} = (X_1,\ldots,X_d)$ in $\bm{S}=(S_1,\ldots,S_d)$, where the $S_i$'s are defined like in Lemma \ref{correspondence}. Then we can see that $\sum_{m =0}^{\infty}V(\bm{S})^m \in M_m(\C) \otimes M_{d}(\C) \otimes M_2(L^{\infty}(\bm{s}))$ converges in the operator norm from Corollary \ref{convergence} since all entries of $\mu(v)$ are given by finite linear spans of $\alpha_{u v w}$ for some words $u,w$. Thus we can conclude 
\begin{eqnarray*}
\sum_{v \in [d]^*} \alpha_v \hat{U}_v &=& \begin{pmatrix} 1 & 0 \end{pmatrix} ({}^t e_1 \otimes I_2)\sum_{v \in [d]^*}\alpha_{v} S^v   (e \otimes I_2) \begin{pmatrix} 1 \\ 0 \end{pmatrix} \Omega \\
&=&  \begin{pmatrix} 1 & 0 \end{pmatrix} ({}^t e_1 \otimes I_2)   {}^t \lambda [1-V(\bm{S})]^{-1} \gamma (e \otimes I_2) \begin{pmatrix} 1 \\ 0 \end{pmatrix} \Omega.
\end{eqnarray*}
Note that $\lim_{m \to \infty} \|V(\bm{S})^m\| = 0$, and we can apply Lemma \ref{rational_entry} to $V(\bm{S})$. Since $\Omega$ is a separating vector for $L^{\infty}(\bm{s})$, we conclude $a \in C_{\mathrm{div}}(\bm{s})$.
\end{proof}

\begin{proof}[Proof of Theorem \ref{main_result}]
 Let $\mathcal{A}$ be a subset of $L^{\infty}(\bm{s})$ such that $\{[r_i^*,a]\}_{i=1}^d$ are finite rank operators on $\mathcal{F}(H)$ for any element $a \in \mathcal{A}$. We will show $\mathcal{A}$ is a subalgebra of $L^{\infty}(\bm{s})$ which contains $\C \langle \bm{s} \rangle$ and satisfies for any $n \in \N,$
$$ X \in M_n(\mathcal{A}) \ \mathrm{is \ invertible \ in }\ M_n(L^{\infty}(\bm{s})) \implies X^{-1} \in M_n(\mathcal{A}).$$
Note that $s_i \in \mathcal{A}$ for any $ i\in [d]$ since $[r_i^*,s_j] = \delta^i_j P_{\Omega}$ is a finite rank operator for any $i,j \in [d]$. If $a, b \in \mathcal{A}$, then the following operators
\begin{eqnarray*}
[r_i^*,a+b] &=& [r_i^*,a] + [r_i^* ,b]  \\
{}[r_i^*, a b] &=& [r_i^*, a] b + a [r_i^*,b]
\end{eqnarray*} 
are finite rank operators for each $i \in [d]$. Thus $\mathcal{A}$ is a subalgebra of $L^{\infty}(\bm{s})$ which contains $\C\langle \bm{s} \rangle$. Let $n \in \N$ be given and assume $X \in M_n(\mathcal{A})$ is invertible in $M_n(L^{\infty}(\bm{s}))$. Then we have for any $i \in [d]$ and $1 \le j,k \le n$
\begin{eqnarray*}
[r_i^*, {}^t e_j X^{-1}e_k] &=& {}^t e_j [I_n \otimes r_i^*, X^{-1}] e_k \\
{}                             &=& - {}^t e_j X^{-1}[I_n \otimes r_i^*, X]X^{-1} e_k 
\end{eqnarray*}
where $I_n \otimes r_i^* \in M_n(\C) \otimes B(\mathcal{F}(H)) \cong M_n(B(\mathcal{F}(H)))$ is the operator such that all diagonal entries are $r_i^*$ and other entries are zero.  
Since $X \in M_n(\mathcal{A})$, all entries of $[I_n \otimes r_i^*, X]$ are finite rank operators and therefore $X^{-1} \in M_n(\mathcal{A})$. Since $C_{\mathrm{rat}}(\bm{s})$ is the smallest subalgebra satisfying above properties, we obtain $C_{\mathrm{rat}}(\bm{s}) \subset \mathcal{A}$. 

Moreover, we have $\mathcal{A} \subset C_{\mathrm{div}}(\bm{s}) \subset \overline{\C \langle \bm{s} \rangle}$ by Proposition \ref{onlyif} and thus $C_{\mathrm{rat}}(\bm{s})=  C_{\mathrm{div}}(\bm{s}) =\mathcal{A} \subset \overline{\C \langle \bm{s} \rangle}$. 
\end{proof}
\begin{rem}\label{for_creation_operator}
  Let us see what happens when we take $r_i$ and consider $[r_i,a]$ instead of $[r_i^*,a]$ for $a \in L^{\infty}(\bm{s})$. 
Indeed, for any $i \in [d]$, $[r_i,a]$ is a finite rank operator if and only if $[r_i^*,a]$ is also a finite rank operator since $[r_i + r_i^*, a] = 0$ and therefore $[r_i,a] = - [r_i^*,a]$ for any $a \in  L^{\infty}(\bm{s})$.
This is deduced from the commutativity of the left multiplication with the right multiplication. One can also see this directly via the following equalities
\begin{eqnarray*}
[r_i + r_i^*, s_j] &=& [r_i, s_j] + [r_i^* , s_j]\\
                           &=& - ([r_i^*, s_j])^* + \delta^i_j P_{\Omega} \\
                         &=&- \delta^i_j P_{\Omega}^* + \delta^i_j P_{\Omega}=0.
\end{eqnarray*}
for any $i,j \in [d]$ where we use $[a,b]^* = b^*a^*-a^* b^* = -[a^*,b^*]$ and $[r_i^* , s_j] = \delta^i_j P_{\Omega}$. Then we have $[r_i + r_i^*,a] = 0$ for any $a \in \C \langle \bm{s} \rangle$ and thus for any $a \in L^{\infty}(\bm{s})$.
\end{rem}
\begin{rem}
We remark that a tuple of operators $(r_1^*, r_2^*,\ldots,r_d^*)$ is known as a dual system which is introduced by D. Voiculescu in \cite{Voi98}.

In our setting, $(D_1,\ldots,D_d)\in B(\mathcal{F}(H))$ is called a dual system for $\bm{s}$ if we have for any $i,j \in [d]$
$$[D_i, s_j] = \delta^i_j P_{\Omega}.$$
From the proof of Lemma \ref{dual_system}, $(r_1^*, r_2^*,\ldots,r_d^*)$ is obviously a dual system and we have 
$$[D_i,a] = [r_i^*,a]$$
for each $i \in [d]$ and $a \in  L^{\infty}(\bm{s})$. Thus Theorem \ref{main_result} holds if we change $\{r_i^*\}_{i=1}^d$ by any dual system for $\bm{s}$.
 
One can also see that $(D_1,\ldots,D_d)\in B(\mathcal{F}(H))$ is a dual system for $\bm{s}$ if and only if 
$r_i^*-D_i$ belongs to $L^{\infty}(\bm{s})' $ for each $i \in [d]$ 
where $L^{\infty}(\bm{s})'$ is the commutant of $L^{\infty}(\bm{s})$. 

We have not proved Theorem \ref{main_result} for a general tuple of operators with a dual system yet and we leave it for future works.  
\end{rem}
\section{Rationality criterion for affiliated operators}\label{Rationality criterion for affiliated operators}
In this section, we extend our main result in the previous section to affiliated operators, which follows  results of Linnell \cite{Lin00}. Let us denote by $\widetilde{L^{\infty}(\bm{s})}$ the $\ast$-algebra of closed densely defined (unbounded) linear operators affiliated with $L^{\infty}(\bm{s})$. Note that any element $u \in \widetilde{L^{\infty}(\bm{s})}$ can be written as $u = f^{-1}a=b g^{-1}$ by using some $a,b,f,g \in L^{\infty}(\bm{s})$ where $f,g$ are nonzero divisors (i.e. $f x, gx \neq 0$ for any $x \in  L^{\infty}(\bm{s})\setminus \{0\})$ and thus invertible in $\widetilde{L^{\infty}(\bm{s})}$. For example, we can take $f = (1+ u u^*)^{-1}, \ a=(1+ u u^{*})^{-1} u, \ b = u (1+ u^*u)^{-1}, \ g = (1+ u^* u)^{-1}$. We focus on bounded operators $\{f r_i^* b - a r_i^* g\}_{i =1}^d$ instead of commutators $\{[r_i^*,u]\}_{i=1}^d$. Note that we have formally $f r_i^* b - a r_i^* g = f[r_i^*,u]g$ since $u = f^{-1}a=b g^{-1}$.

The following lemma tells us that we can find a common denominator of two affiliated operators. 
\begin{lm}\label{common_denominator}
Let $u_1,u_2 \in \widetilde{L^{\infty}(\bm{s})}$.  Then there exist $a_1,a_2,b_1,b_2 \in L^{\infty}(\bm{s})$ and $f, g \in L^{\infty}(\bm{s})$ such that $u_1 = f^{-1}a_1 = b_1 g^{-1}$ and $u_2 = f^{-1} a_2 = b_2 g^{-1}$. 
\end{lm} 
\begin{proof}
Let $u_k = f_k^{-1}a_k = b_k g_k^{-1}$ for $k=1,2$ where $a_k,b_k,f_k,g_k \in \widetilde{L^{\infty}(\bm{s})}$. Then we can write $f_1 f_2^{-1} = x^{-1} y$ for some $x,y \in  L^{\infty}(\bm{s})$. Note that $f_1^{-1} = (y f_2)^{-1} x$ and $f_2^{-1} = (x f_1)^{-1}y$ and $x f_1 = y f_2$. We put $f = x f_1 = y f_2$.
Then we have $u_1 = f^{-1} x a_1$ and $u_2 = f^{-1} y a_2$. Similarly by representing $g_1^{-1}g_2 = x y^{-1}$ for some $x,y \in L^{\infty}(\bm{s})$, we have $u_1 = b_1 x g^{-1}$ and $u_2 = b_2 y g^{-1}$ where $g = g_1 x  = g_2 y$.
\end{proof}
We also use the following lemmas for bounded operators and affiliated operators (see \cite{Lin00}).
\begin{lm}[Lemma 2.1 in \cite{Lin00}] \label{lemma1}
Let $\theta:H \to K$ and $\phi : K \to H$ be bounded linear maps between Hilbert spaces.
\begin{enumerate}
\item If $\ker \phi = \{0\}$ and $\phi \theta$ has finite rank, then $\theta$ also has finite rank.\\
\item If $\Ima \theta$ is dense in $K$ and $\phi \theta$ has a finite rank, then $\phi$ also has a finite rank.
\end{enumerate}
\end{lm}
The following lemma is proved in Lemma 2.2 in \cite{Lin00} in terms of the free group, and the proof can be also applied to our setting.  

\begin{lm}[Lemma 2.2 in \cite{Lin00}]\label{lemma2}
Let $\theta \in L^{\infty}(\bm{s})$. If $\theta$ is a nonzero divisor, then $\ker \theta = \{0\}$ and $\Ima \theta $ is dense in $\mathcal{F}(H)$. 
\end{lm}

We define $R(\bm{s})$ and $R'(\bm{s})$ as subsets of $\widetilde{L^{\infty}(\bm{s})}$.
We say $u \in R(\bm{s})$ if $\{f r_i^* b - a r_i^* g\}_{i =1}^d$ are finite rank operators for any expression $u = f^{-1}a = b g^{-1}$ where $a,b,f,g \in L^{\infty}(\bm{s})$. We say $u \in R'(\bm{s})$ if we can write $u = f^{-1}a = b g^{-1}$ for some $a,b,f,g \in L^{\infty}(\bm{s})$ such that  $\{f r_i^* b - a r_i^* g\}_{i =1}^d$ 
are finite rank operators. Note that we have $R(\bm{s}) \subset R'(\bm{s})$ by definition.

To define rationality, we consider the division closure $D(\bm{s})$ of $\C \langle \bm{s} \rangle$ in $\widetilde{L^{\infty}(\bm{s})}$. From the results in \cite{MSY19}, $D(\bm{s})$ forms the free skew field of fractions of $\C\langle\bm{s}\rangle$. Note that the rational closure of $\C \langle \bm{s} \rangle$ in $\widetilde{L^{\infty}(\bm{s})}$ coincides with $D(\bm{s})$ since $D(\bm{s})$ is a skew field (see \cite[Proposition 4.9]{MSY19}).
\begin{rem}\label{creation_affiliate}
Let $i \in [d]$ and $u= f^{-1}a=b g^{-1} \in \widetilde{L^{\infty}(\bm{s})}$ where $a,b,f,g \in L^{\infty}(\bm{s})$ and assume $f r_i^* a - b r_i^* g$ is a finite rank operator. Then thanks to Remark \ref{for_creation_operator}, we have  
$$f(r_i^* + r_i) b - a (r_i^*+r_i)g = (f b- a g) (r_i^*+r_i)=0$$
for any $i \in [d]$ where we use $f b= a g$. Thus $f r_i b- a r_i g$ is also a finite rank operator.  
\end{rem}
Let us state the main theorem in this section. 
\begin{thm}\label{Linnell}
We have $R(\bm{s}) = R'(\bm{s}) = D(\bm{s})$ and $D(\bm{s}) \cap L^{\infty}(\bm{s}) = C_{\mathrm{div}}(\bm{s}).$ Moreover, for any $u \in D(\bm{s})$, there exists $a,b,f,g \in C_{\mathrm{div}}(\bm{s})$ such that $u = f^{-1} a = b g^{-1}$.   
\end{thm}
In order to prove this theorem, first we show $R(\bm{s}) = R'(\bm{s})$.
\begin{lm}\label{specific_choice}
Let $u \in \widetilde{L^{\infty}(\bm{s})}$ and assume $u = f^{-1} a = b g^{-1}$. If $\{f r_i^* b - a r_i^* g\}_{i=1}^d$   
are finite rank operators, then $u \in R(\bm{s})$.
In other words, we have $R(\bm{s}) = R'(\bm{s})$.  
\end{lm}
\begin{proof}
Let $u=f_1^{-1} a_1 = b_1 g_1^{-1}$ where $a_1,b_1,f_1,g_1 \in L^{\infty}(\bm{s})$. We need to show 
$f_1 r_i^* b_1 - a_1 r_i^* g_1$ 
is a finite rank operator for any $i \in [d]$. First we note that there exist $x , y \in  L^{\infty}(\bm{s})$ such that $f f_1^{-1} = x^{-1}y$. We infer that  $x f = y f_1$ and $y a_1 = x f f_1^{-1} a_1 = x f f^{-1} a = x a $. Thus we obtain
$$y (f_1 r_i^* b_1 - a_1 r_i^* g_1) = y f_1 r_i^* b_1 - y a_1 r_i^* g_1 = x (f  r_i^* b_1 - a r_i^* g_1). $$
On the other hand, since there exist some $x',y' \in  L^{\infty}(\bm{s})$ such that $g^{-1}g_1 = x' y'^{-1}$, we obtain in the same way as before that
$$(f  r_i^* b_1 - a r_i^* g_1) y'= (f  r_i^* b - a r_i^* g) x'.$$
By combining them, we have
\begin{eqnarray*}
y (f_1 r_i^* b_1 - a_1 r_i^* g_1) y'  &=& x (f  r_i^* b_1 - a r_i^* g_1) y' \\
                                        &=& x (f  r_i^* b - a r_i^* g) x'.
\end{eqnarray*}
Since $f  r_i^* b - a r_i^* g$ is a finite rank operator for any $i \in [d]$ and $y,y'$ are non-zero divisors, $f_1 r_i^* b_1 - a_1 r_i^* g_1$ is also a finite rank operator for any $i \in [d]$ by Lemmas \ref{lemma1} and \ref{lemma2}; hence, we see that $u \in R(\bm{s})$. This shows $R'(\bm{s}) \subset R(\bm{s})$ and thus we conclude $R(\bm{s}) = R'(\bm{s})$.   
\end{proof}
\begin{rem}\label{boundedcase}
If $u \in R(\bm{s})\cap L^{\infty}(\bm{s})$, since we can write $u = u 1^{-1}=1^{-1}u$, $\{[r_i^*,u]\}_{i=1}^d$ are finite rank operators. Thanks to Theorem \ref{main_result}, we have $u \in C_{\mathrm{div}}(\bm{s})$. On the other hand, if $u \in C_{\mathrm{div}}(\bm{s})$, then $\{[r_i^*,u]\}_{i=1}^d$ are finite rank operators, and thus $u \in R'(\bm{s})$ by the same theorem. By using Lemma \ref{specific_choice}, we have    
$$R(\bm{s})\cap L^{\infty}(\bm{s}) = R'(\bm{s}) \cap L^{\infty}(\bm{s}) = C_{\mathrm{div}}(\bm{s}).$$ 
\end{rem}

We will prove four lemmas in order to deduce that $R(\bm{s})$ is a $*$-subalgebra which is closed under taking inverse.  
\begin{lm}\label{addition}
If $u_1, u_2 \in R(\bm{s})$, then $u_1 + u_2 \in R(\bm{s}).$
\end{lm}
\begin{proof}
By Lemma \ref{common_denominator}, we can write $u_k = f^{-1} a_k = b_k g^{-1}$ for $k=1,2$. Then $u_1 + u_2 = f^{-1} (a_1 + a_2) = (b_1 + b_2) g^{-1}$. Since $u_1,u_2 \in R(\bm{s})$ and for any $i \in [d]$ 
$$f r_i^* (b_1 + b_2) - (a_1+a_2) r_i^* g = (f r_i^* b_1 - a_1 r_i^* g) + (f r_i^* b_2 - a_2 r_i^* g),$$
we see that $f r_i^* (b_1 + b_2) - (a_1+a_2) r_i^* g$ is a finite rank operator for any $ i\in [d]$. 
 Therefore $u_1 + u_2 \in R(\bm{s})$ by Lemma \ref{specific_choice}.    
\end{proof}
\begin{lm}\label{multiplication}
If $u_1, u_2 \in R(\bm{s})$, then $u_1 u_2 \in R(\bm{s})$.
\end{lm}
\begin{proof}
Let us write $u_k = f_k^{-1} a_k = b_k g_k^{-1}$ where $a_k,b_k,f_k,g_k \in L^{\infty}(\bm{s})$ for $ k=1,2$. Let $a_1 f_2^{-1} = x^{-1} y$ and $g_1^{-1}b_2 = p q^{-1}$ where $p,q,x,y \in L^{\infty}(\bm{s})$. Then $u_1u_2 = f_1^{-1} a_1 f_2^{-1} a_2 = (x f_1)^{-1} y a_2$ and $u_1 u_2 = b_1 g_1^{-1} b_2 g_2^{-1} = b_1 p (g_2q)^{-1}$. Since $x a_1 = y f_2$ and $g_1 p = b_2 q$, we have 
$$ x f_1 r_i^* b_1 p - y a_2  r_i^* g_2 q = x ( f _1 r_i^* b_1 - a_1 r_i^* g_1) p + y ( f_2 r_i^* b_2 - a_2 r_i^* g_2)q.$$
Since $u_1,u_2 \in R(\bm{s})$, this operator is a finite rank operator for any $i \in [d]$,  
and thus $u_1 u_2 \in R(\bm{s})$ by Lemma \ref{specific_choice}. 
\end{proof}
\begin{lm}\label{inversion}
If $u \in R(\bm{s})$ is invertible, then $u^{-1} \in R(\bm{s})$.  
\end{lm}
\begin{proof}
If $u \in R(\bm{s})$, then we can write $u = f^{-1}a = b g^{-1}$ and $f r_i^* b - a r_i^* g$ has a finite rank for any $i \in [d]$. In addition if $u$ is invertible, we have $u^{-1} =a^{-1} f = g b^{-1}$. Since $a r_i^* g - f r_i^* b = - (f r_i^* b - a r_i^* g)$  
for each $i \in [d]$, $u \in R(\bm{s})$ by Lemma \ref{specific_choice}.  
\end{proof}
\begin{lm}\label{involution}
If $u \in R(\bm{s})$, then $u^* \in R(\bm{s})$.  
\end{lm}
\begin{proof}
If $u \in R(\bm{s})$, then we can write $u = f^{-1}a = b g^{-1}$ and $f r_i^* b -a r_i^* g$ is a finite rank operator for any $i \in [d]$. Since $u^* = {g^*}^{-1}b^* = a^* {f^*}^{-1}$, we need to check that $g^* r_i^* a^* - b^* r_i^* f^*$ is a finite rank operator.

Since $T^*$ is a finite rank operator if $T$ is a finite rank operator on a Hilbert space and $f r_i b -a r_i g$ is a finite rank operator by Remark \ref{creation_affiliate}, $g^* r_i^* a^* - b^* r_i^* f^* = -(f r_i b - a r_i g)^*$ is also a finite rank operator for any $i \in [d]$.  
Thus we conclude $u^* \in R(\bm{s})$ by Lemma \ref{specific_choice}. 
\end{proof}
\begin{proof}[Proof of Theorem \ref{Linnell}]
By Lemmas \ref{addition}, \ref{multiplication}, \ref{inversion}, we see that $R(\bm{s})$ is a subalgebra of $\widetilde{L^{\infty}(\bm{s})}$ which contains $\C \langle \bm{s} \rangle $ and is closed under taking inverse. Thus $D(\bm{s}) \subset R(\bm{s})$. 

Now, let $u \in R(\bm{s})$. Since $R(\bm{s})$ is also closed under the involution by Lemma \ref{involution}, $a = (1 + u u^*)^{-1}u$ and $f = (1+u u^*)^{-1}$ belong to $R(\bm{s})\cap L^{\infty}(\bm{s}) = C_{\mathrm{div}}(\bm{s})$ (see Remark \ref{boundedcase}) and therefore $u = f^{-1} a$ belongs to the division closure of $C_{\mathrm{div}}(\bm{s})$ in $\widetilde{L^{\infty}(\bm{s})}$. Since $D(\bm{s})$ is the division closed subalgebra of  $\widetilde{L^{\infty}(\bm{s})}$ which contains $C_{\mathrm{div}}(\bm{s})$, it also contains the division closure of $C_{\mathrm{div}}(\bm{s})$ in $\widetilde{L^{\infty}(\bm{s})}$ (both coincide actually). Thus $u \in D(\bm{s})$.   
\end{proof}

In Theorem \ref{Linnell}, we show an equivalent condition to $u \in D(\bm{s})$ by using bounded operators $\{f r_i^* b - a r_i^* g\}_{i =1}^d$ instead of commutators $\{[r_i^*,u]\}_{i=1}^d$. As we remark in the beginning of Section \ref{Rationality criterion for affiliated operators}, both operators $f r_i^* b - a r_i^* g$ and $[r_i^*,u]$ are formally connected by $f r_i^* b - a r_i^* g = f[r_i^*,u]g$.

In the following proposition, we give another characterization of $u \in D(\bm{s})$ by using commutators $\{[r_i^*,u]\}_{i=1}^d$, which is an analogue of Proposition 1.2 in \cite{Lin00}.    
\begin{pr}\label{unbounded_commutator}
Let $u \in \widetilde{L^{\infty}(\bm{s})}$. Then $u \in D(\bm{s})$ if and only if there exists a linear subspace $M$ of finite codimension in $\mathcal{F}(H)$ such that $M \cap \bigcap_{i \in [d]} \dom(u r_i^*)=M \cap \dom(u) $ and $r^*_i u = u r_i^*$ on $M \cap \dom(u)$ for each $i \in [d]$, where $\dom(u)$ denotes the domain of $u$.     
\end{pr}

\begin{proof}
 We use the well-known fact that for any subspace $M$ of finite codimension in a linear space $H$ and for any linear map $T$ on $H$, the preimage $T^{-1}(M)$ is also a subspace of finite codimension in $H$ (since $T$ induces an injective linear map from the quotient subspace $H/T^{-1}(M)$ to $H/M$ which is finite-dimensional). 

In addition, an intersection $M_1 \cap M_2$ of two subspaces $M_1, M_2$ of finite codimension in $H$ is also a subspace of finite codimension (since $(M_1+M_2)/M_2$ is isomorphic to $M_1/(M_1 \cap M_2)$ and the two quotient spaces $H/M_1$, $(M_1+M_2)/M_2$ are finite-dimensional).  

Now we suppose $M$ is a subspace of finite codimension such that $r^*_i u= u r_i^*$ on $M\cap \dom(u)=M \cap \bigcap_{i \in [d]} \dom(u r_i^*)$ for any $i \in [d]$. We can write $u$ as $u =f^{-1}a= b g^{-1}$ where $a,b,f,g \in L^{\infty}(\bm{s})$. Note that $N = g^{-1}(M)$ is a subspace of finite codimension in $\mathcal{F}(H)$ such that $g N \subset M \cap \dom(u) = M \cap \bigcap_{i \in [d]} \dom (u r_i^*).$ Thus we have for any $\xi \in N$ 

$$(f r_i^* b   - a r_i^* g) \xi= f (r_i^*u g \xi - u r_i^* g \xi)  =0.$$
Since $N$ has a finite codimension in $\mathcal{F}(H)$, $f r_i^* b - a r_i^* g$ is a finite rank operator for each $i \in [d]$ and thus $u \in D(\bm{s})$ by Theorem \ref{Linnell}.  

On the other hand, if $u \in D(\bm{s})$, then by Theorem \ref{Linnell} there exists $a, f \in C_{\mathrm{div}}(\bm{s})$ such that $u = f^{-1} a$. Note that $f^{-1} a$ forms a closed operator even though we see it as a composition of unbounded operators (we do not have to take closure). Thus we can write $\dom(u) = \{\xi \in \mathcal{F}(H) ; a\xi \in f \mathcal{F}(H)\}$.
Since $a, f \in C_{\mathrm{div}}(\bm{s})$, $[r_i^*, f], \ [r_i,f], \ [r_i^*,a]$ are finite rank operators for each $i \in [d]$. Then kernels of these operators have finite codimensions.
We put for each $i \in [d]$

$$N_i = \ker[r_i^*,f] \cap \ker[r_i, f]$$
 and define $N$ as
$$N = \bigcap_{i \in [d]} N_i. $$
    Note that $N$ is a subspace of finite codimension in $\mathcal{F}(H)$ and there exists a subspace $M_1$ of finite codimension in $\mathcal{F}(H)$ such that   

$$M_1 \cap f \mathcal{F}(H) \subset f N.$$ 
For example, we can take $M_1$ as a direct sum of $f N$ and a complementary subspace of $f\mathcal{F}(H)\subset \mathcal{F}(H)$.
Since $r_i f N = f r_i N$ for each $i \in [d]$, there exists a subspace $M_2^i$ of finite codimension in $\mathcal{F}(H)$ for each $i \in [d]$ such that

$$M_2^i \cap r_i f \mathcal{F}(H) \subset f \mathcal{F}(H).$$

We can take $M_2^i$ either as a direct sum of $r_i f N$ and complementary subspace of $f\mathcal{F}(H) \subset \mathcal{F}(H)$ as above, or as $(r_i^*)^{-1}(M_1)$.
  Then we put $M$ by
$$M = a^{-1}[(\C \Omega)^{\perp}] \cap a^{-1}(M_1) \cap \bigcap_{i \in [d]} \ker[r_i^*,a] \cap \bigcap_{i \in [d]} (r_i r_i^* a)^{-1} (M_2^i).$$ 
Then $M$ is obviously a subspace of finite codimension in $\mathcal{F}(H)$. 

Let us show 
 $M \cap \bigcap_{i \in [d]} \dom(u r_i^*)=M \cap \dom(u). $
If $\xi \in M \cap \dom(u)$, then $a \xi = f \eta$ for $\eta \in N$ since $a \xi \in M_1 \cap f \mathcal{F}(H) \subset f N$. Since $\xi \in \ker[r_i^*,a]$ and $\eta \in N$, we obtain $a r_i^* \xi = f r_i^* \eta$ which implies $\xi \in \dom(u r_i^*)$ for any $i \in [d]$. On the other hand, if $\xi \in M \cap  \bigcap_{i \in [d]} \dom(u r_i^*)$, then $r_i^* a \xi=a r_i^* \xi \in f \mathcal{F}(H)$ for each $i \in [d]$.
By multiplying by $r_i$, we have $r_i r_i^* a \xi \in M_2^i \cap r_i f \mathcal{F}(H) \subset f \mathcal{F}(H)$ for each $i \in [d]$. Then there exists $\eta_i \in \mathcal{F}(H)$ for each $i \in [d]$ such that $r_i r_i^* a \xi = f \eta_i$. Since $a \xi \in (\C \Omega)^{\perp}$, we have 
$$a\xi = \sum_{i \in [d]} r_i r_i^* a \xi = f \sum_{i \in [d]} \eta_i $$
which implies $\xi \in \dom(u)$. Therefore we have $M \cap \bigcap_{i \in [d]} \dom(u r_i^*)=M \cap \dom(u). $

In order to see $u r_i^* = r_i^* u$ on $M \cap \dom(u)$, we take $\xi \in M \cap \dom(u)$. Then, as shown above, there exists $\eta \in N$ such that $a \xi  = f \eta$; hence $r_i^*u \xi = r_i^* \eta$ for each $i \in [d]$. Moreover since $a r_i^* \xi = f r_i^* \eta$, as shown above, we conclude $u r_i^* \xi = r_i^* \eta = r_i^* u \xi$.  
\end{proof}
\begin{rem}
The proof of Proposition \ref{unbounded_commutator} also works when we take $\{r_i\}_{i \in [d]}$ instead of $\{r_i^*\}_{i \in [d]}$. In this case, we can say that there exists a subspace of finite codimension in $\mathcal{F}(H)$ such that $M \cap \dom(u) = M \cap \dom(u r_i)$ and $u r_i = r_i u$ on $M \cap \dom(u)$ for each $i \in [d]$.
\end{rem}
\bibliographystyle{amsalpha}

\end{document}